\documentclass[12pt,reqno]{amsart}
\usepackage{amsaddr}
\usepackage[hidelinks]{hyperref}
\author{Tirthankar Bhattacharyya and Abhay Jindal}
\address{Department of Mathematics\\
	Indian Institute of Science\\
	Bangalore 560012, India}
\email{tirtha@iisc.ac.in; abjayj@iisc.ac.in}

\usepackage{color, bm, amscd, tikz-cd}
\usepackage{csquotes}
\newcommand{\bydef}{\stackrel{\rm def}{=}}

\usepackage{tikz}

\newcommand{\cB}{{\mathcal B}}

\newcommand{\cD}{{\mathcal D}}
\newcommand{\cE}{{\mathcal E}}

\newcommand{\cH}{{\mathcal H}}

\newcommand{\cM}{{\mathcal M}}

\newcommand{\la}{\langle}
\newcommand{\ra}{\rangle}

\newcommand{\bfT}{\textit{\textbf{T}}}
\newcommand{\bfR}{\textit{\textbf{R}}}
\newcommand{\bfM}{\textit{\textbf{M}}}

\newcommand{\bfZ}{\textit{\textbf{Z}}}

\newtheorem{thm}{Theorem}[section]

\newtheorem{corollary}[thm]{Corollary}
\newtheorem{lemma}[thm]{Lemma}

\newtheorem{proposition}[thm]{Proposition}

\newtheorem{definition}[thm]{Definition}

\numberwithin{equation}{section}

\def\textmatrix#1&#2\\#3&#4\\{\bigl({#1 \atop #3}\ {#2 \atop #4}\bigr)}
\def\dispmatrix#1&#2\\#3&#4\\{\left({#1 \atop #3}\ {#2 \atop #4}\right)}
\numberwithin{equation}{section}

\def\textmatrix#1&#2\\#3&#4\\{\bigl({#1 \atop #3}\ {#2 \atop #4}\bigr)}
\def\dispmatrix#1&#2\\#3&#4\\{\left({#1 \atop #3}\ {#2 \atop #4}\right)}

\begin{document}
	
\title{Complete Nevanlinna-Pick kernels and the curvature invariant}
	
\maketitle
\begin{abstract}
We consider a unitarily invariant complete Nevanlinna-Pick kernel denoted by $s$ and a commuting $d$-tuple of bounded operators $\bfT = (T_{1}, \dots, T_{d})$ satisfying a natural contractivity condition with respect to $s$. We associate with $\bfT$ its curvature invariant which is a non-negative real number bounded above by the dimension of a defect space of $\bfT$. The instrument which makes this possible is the characteristic function developed in \cite{BJ}. 
	
\medskip
	
We present an asymptotic formula for the curvature invariant. In the special case when $\bfT$ is pure, we provide a notably simpler formula, revealing that in this instance, the curvature invariant is an integer. We further investigate its connection with an algebraic invariant known as fibre dimension. Moreover, we obtain a refined and simplified asymptotic formula for the curvature invariant of $\bfT$ specifically when its characteristic function is a polynomial.
\end{abstract}

{\footnotesize \noindent 2020 Mathematics Subject Classification: 47A13, 47A15, 46E22. \\
Keywords: Complete Nevanlinna-Pick kernels, Dirichlet kernel, Drury-Arveson kernel, Curvature invariant, Fibre dimension}

\section{Introduction}

A reproducing kernel $s$ on the open Euclidean unit ball 
$$\mathbb{B}_{d} = \{ \bm z = (z_{1}, \dots ,z_{d}) \in \mathbb{C}^{d}: \|\bm z \| \bydef (|z_{1}|^{2} + \dots + |z_{d}|^{2})^{1/2} <1\}$$
is said to be a complete  Nevanlinna-Pick (CNP) kernel if for any natural numbers $m,n,$ any $N$ points $\bm \lambda_{1}, \dots, \bm \lambda_{N}$ in $\mathbb{B}_{d}$ and any $m \times n$ matrices $W_{1}, \dots, W_{N},$ the condition that the $N \times N$ block matrix
$$\Big( (I - W_{i}W_{j}^{*}) s(\bm \lambda_{i}, \bm \lambda_{j}) \Big)_{i,j=1}^{N}$$
is positive semidefinite implies that there is a holomorphic function $\varphi: \mathbb{B}_{d} \to \mathbb{M}_{m \times n} (\mathbb{C})$ mapping $H_{s} \otimes \mathbb{C}^{n}$ into $H_{s} \otimes \mathbb{C}^{m}$ by multiplication and interpolating $\bm \lambda_{i}$ to $W_{i}, i=1, \dots, N.$ The relation to the classical Nevanlinna-Pick interpolation problem is not hard to imagine, hence the name.

\medskip
 
Consider a reproducing kernel 
\begin{equation}\label{unitarily invariant}
s: \mathbb{B}_{d} \times \mathbb{B}_{d} \to \mathbb{C}; \quad s(\bm z, \bm w) = \sum\limits_{n=0}^{\infty} a_{n} \la \bm z , \bm w \ra^{n}, \quad (\bm z ,\bm w \in \mathbb{B}_{d})
\end{equation}
with $a_{0} = 1$ and $a_{n} > 0$ for all $n \geq 1.$  It is clearly {\em unitarily invariant}, i.e.,
$$s(U \bm z, U \bm w) = s(\bm z, \bm w) \quad \text{for all $d\times d$ unitary matrices $U$}.$$
It is called irreducible if $s (\bm z, \bm w) \neq 0$ for all $ \bm z, \bm w \in \mathbb{B}_{d}$ and $s_{\bm w}$ and $s_{\bm \nu}$ are linearly independent if $\bm \nu \neq \bm w$ where $s_{\bm w}(\bm z) = s(\bm z, \bm w).$ 

\medskip

There is a characterization which we shall greatly use, viz., a reproducing kernel $s$ defined in \eqref{unitarily invariant} is irreducible and CNP  if and only if there is a sequence of non-negative real numbers $\{b_{n}\}_{n=1}^{\infty}$ such that
\begin{equation} \label{bn}
1 - \frac{1}{s(\bm z, \bm w)} = \sum\limits_{n=1}^{\infty} b_{n} \la \bm z, \bm w \ra^{n}, \quad (\bm z, \bm w \in \mathbb{B}_{d}).
\end{equation}
In this note, $s$ will always stand for a regular unitarily invariant CNP kernel defined below.
\begin{definition}\label{def}
A reproducing kernel $s: \mathbb{B}_{d} \times \mathbb{B}_{d} \to \mathbb{C}$ defined by $\eqref{unitarily invariant}$ is called a regular unitarily invariant CNP kernel if 
\begin{enumerate}
	\item it is an irreducible CNP kernel, or equivalently, there exists a sequence of non-negative real numbers $\{b_{n}\}_{n=1}^{\infty}$ satisfying \eqref{bn},
	\item  $\lim\limits_{n \rightarrow \infty} \frac{a_{n}}{a_{n+1}}  = 1,$ and
	\item $\sum\limits_{n=0}^{\infty} a_{n} = \infty$ (equivalently $ \sum\limits_{n=1}^{\infty} b_{n} =1$ or $k_{\bm z}(\bm z) \rightarrow \infty$ as $|\bm z| \rightarrow 1.$)
\end{enumerate}
The corresponding reproducing kernel Hilbert space is denoted by $H_{s}.$ 
\end{definition}

	 Denote by $\mathbb{Z}_{+}$ the set of all non-negative integers. Let $\alpha= (\alpha_{1}, \ldots ,\alpha_{d})\in\mathbb{Z}^{d}_{+}$ be a {\em multi-index}. Let $\bm{z}\in \mathbb{C}^{d}$. We need the following notations.
	$$ |\alpha|=\alpha_{1}+ \cdots +\alpha_{d}, \;\; \alpha !=\alpha_{1}! \ldots \alpha_{d}!, \;\; \binom{|\alpha|}{\alpha} = \frac{|\alpha|!}{\alpha_{1}!\dots\alpha_{d}!} \text{ and } \bm z^{\alpha} = z_{1}^{\alpha_{1}}  \ldots z_{d}^{\alpha_{d}}.$$
	To simplify notations, we define the coefficients $a_{\alpha}$ and $b_{\alpha}$ as follows:
	\begin{equation*}
		a_{\alpha} = \begin{cases}  a_{|\alpha|} \binom{|\alpha|}{\alpha}, & \alpha\in\mathbb{Z}^{d}_{+} \\ 0, & \alpha\in\mathbb{Z}^{d} \backslash \mathbb{Z}^{d}_{+} \end{cases},\quad  \text{and} \quad
		b_{\alpha}= b_{|\alpha|} \binom{|\alpha|}{\alpha},\quad  \alpha\in\mathbb{Z}^{d}_{+} \backslash \{0\}.
	\end{equation*}
We shall need, for each multi-index $\alpha\in\mathbb{Z}^{d}_{+} \backslash \{0\},$  the polynomial
\begin{equation} \label{psi} 
\psi_{\alpha} : \mathbb{B}_{d} \to \mathbb{C}; \quad \bm z \mapsto (b_{\alpha})^{1/2} \bm{z}^{\alpha}.\end{equation}

\medskip

We now turn to $d-$tuples of bounded operators $ \bfT = (T_{1}, \dots, T_{d})$. Any tuple of bounded operators in this note always consists of commuting operators. Set 
$$\bfT^{\alpha} =  T_{1}^{\alpha_{1}} \dots T_{d}^{\alpha_{d}}, \quad \alpha \in \mathbb{Z}^{d}_{+}.$$
The following definition inspired by the expression in \eqref{bn} is introduced in \cite{CH} and plays a vital role in our analysis. 

\begin{definition}
If the series $\sum\limits_{\alpha\in\mathbb{Z}^{d}_{+} \backslash \{0\}} b_{\alpha} \bfT^{\alpha}(\bfT^\alpha)^{*}$ converges in the strong operator topology to a contraction, then the $d-$tuple $\bfT$ is referred to as a $1/s$-contraction.
In this case, we denote the unique positive square root of the positive operator
$$ I-\sum\limits_{\alpha\in\mathbb{Z}^{d}_{+} \backslash \{0\}} b_{\alpha} \bfT^{\alpha}(\bfT^{\alpha})^{*}$$
by $\Delta_{\bfT}.$  We shall call $\Delta_{\bfT}$ the  defect operator.
\end{definition}

When $s$ is the Drury-Arveson kernel $1/(1 - \langle \bm z, \bm w \rangle)$, the defect operators $\Delta_{\bfT}$ is $(I - T_1T_1^* - \cdots - T_dT_d^*)^{1/2}$.

\medskip

The purpose of this note is to shed light on asymptotic formulae for the curvature invariant $ K_{\bfT}$ (to be defined in the context later) for an $1/s$-contraction $\bfT$. To highlight, we choose Theorem \ref{thm1}, where we give an asymptotic formula for $K_{\bfT}$.

\section{The curvature invariant}

One of the principal tools we shall use is the {\em characteristic function} developed in \cite{BJ}. We need to recall the construction. Let
 $$\tilde{\cH} \bydef \oplus_{\alpha\in\mathbb{Z}^{d}_{+} \backslash \{0\}} \cH,$$
 the infinite direct sum of the Hilbert space $\cH.$  Recall the $\psi_\alpha$ from \eqref{psi} and define the infinite operator tuple
	\begin{equation*}\label{def Z}
		\bfZ = (\psi_{\alpha}(z) I_{\cH})_{\alpha\in\mathbb{Z}^{d}_{+} \backslash \{0\}}.
	\end{equation*}
	The same notation $\bfZ$ also serves for the operator from $\tilde{\cH}$ to $\cH$ mapping $(h_{\alpha})_{\alpha\in\mathbb{Z}^{d}_{+} \backslash \{0\}}$ to $\sum\limits_{\alpha\in\mathbb{Z}^{d}_{+} \backslash \{0\}} (b_{\alpha})^{1/2} \bm{z}^{\alpha} h_{\alpha}.$
	Similarly,  $\tilde{\bfT}$ stands for the infinite operator tuple
	\begin{equation*}\label{def Ttilde}
		\tilde{\bfT}  =  (\psi_{\alpha}(\bfT))_{\alpha\in\mathbb{Z}^{d}_{+} \backslash \{0\}}
	\end{equation*}
	as well as the operator from $\tilde{\cH}$ to $\cH$ which maps $(h_{\alpha})_{\alpha\in\mathbb{Z}^{d}_{+} \backslash \{0\}}$ to $\sum\limits_{\alpha\in\mathbb{Z}^{d}_{+} \backslash \{0\}} (b_{\alpha})^{1/2} \bfT^{\alpha} h_{\alpha}.$ Note that
	$$\|\bfZ\|^{2} = \sum\limits_{\alpha\in\mathbb{Z}^{d}_{+} \backslash \{0\}} b_{\alpha} | \bm{z}^{\alpha}|^{2} = 1 - \frac{1}{s(\bm{z}, \bm{z})} < 1.$$
Moreover, $\tilde{\bfT}$ is a contraction if and only if $\bfT$ is a $1/s$-contraction.  So, $I_{\cH}-\bfZ \tilde{\bfT}^{*}$ is invertible.

\medskip

A straightforward computation shows that $\Delta_{\bfT}^{2} = I_{\cH} - \tilde{\bfT} \tilde{\bfT}^{*}.$ Let $D_{\tilde{\bfT}}$ be the unique positive square root of the positive operator $I_{\tilde{H}} - \tilde{\bfT}^{*} \tilde{\bfT},$ and let $\cD_{\tilde{\bfT}} = \overline{\rm Ran} D_{\tilde{\bfT}}.$ By equation (I.3.4) of \cite{NF} we obtain the identity
	\begin{equation}\label{defect}
		\tilde{\bfT}D_{\tilde{\bfT}}= \Delta_{\bfT} \tilde{\bfT}.
	\end{equation}
	
	\begin{definition}
		The characteristic function of a $1/s$-contraction $\bfT = (T_{1}, \dots, T_{d})$ is the analytic operator valued function $\theta_{\bfT}:\mathbb{B}_{d}\to\cB(\cD_{\tilde{\bfT}}, \overline{\rm Ran} \Delta_{\bfT})$  given by
		$$	\theta_{\bfT}(\bm{z})=(-\tilde{\bfT}+ \Delta_{\bfT}(I_{\cH}-\bfZ \tilde{\bfT}^{*})^{-1}\bfZ D_{\tilde{\bfT}})|_{\cD_{\tilde{\bfT}}}. $$
	\end{definition}
	
	The characteristic function $\theta_{\bfT}$ takes values in $\cB(\cD_{\tilde{\bfT}}, \overline{\rm Ran} \Delta_{\bfT})$ by virtue of \eqref{defect}. It generalizes the characteristic function developed in \cite{BES} for  the special case when $s$ is the Drury-Arveson kernel. For Hilbert spaces $\cE$ and $\mathcal F$, let ${\rm Mult}(H_{s} \otimes \cE, H_{s} \otimes \mathcal F )$ denote the {\em multiplier space}, i.e., the vector space of all $\mathcal B (\cE, \mathcal F)$ valued functions $\varphi$ on $\mathbb B_d$ such that the multiplication operator $M_\varphi$ is in $\mathcal B (H_{s} \otimes \cE, H_{s} \otimes \mathcal F )$. The characteristic function  $\theta_{\bfT}$ is  in ${\rm Mult}(H_{s} \otimes \cD_{\tilde{\bfT}}, H_{s} \otimes \overline{\rm Ran} \Delta_{\bfT} )$ and is contractive because of Theorem 4.11 in \cite{BJ}.
	
\medskip

Let  $\sigma$ denote the normalized measure on the sphere $\partial \mathbb{B}_{d}$. Since $\theta_{\bfT}$ is a bounded analytic  operator-valued function defined in the open unit ball $\mathbb{B}_{d},$ for $\sigma-$almost every $\bm{z} \in \partial \mathbb{B}_{d},$ the radial limits
	\begin{equation*}
		\tilde{\theta_{\bfT}} (\bm{z}) = \lim\limits_{r \uparrow 1} \theta_{\bfT}(r \bm{z}), \quad  {\rm and} \quad \tilde{\theta_{\bfT}}(\bm{z})^{*} = \lim\limits_{r \uparrow 1} \theta_{\bfT}(r \bm{z})^{*}
	\end{equation*}
	exist as strong limits of operators.

\medskip

\textbf {Assumption:} For this note, $\rm dim ({\rm Ran} \Delta_{\bfT})$ is finite.

\begin{definition}
 The curvature function is defined for almost everywhere  $\bm{z} \in \partial \mathbb{B}_{d}$ by
$$ 	K_{\bfT}(\bm{z}) =  \rm dim ({\rm Ran} \Delta_{\bfT}) - {\rm trace}\left( \tilde{\theta_{\bfT}}(\bm{z}) \tilde{\theta_{\bfT}} (\bm{z})^{*}  \right) .$$
The curvature invariant of $\bfT$ is defined by averaging $K_{\bfT}(\cdot)$ over the sphere
$$	K_{\bfT} = \int_{\partial \mathbb{B}_{d}} K_{\bfT}(\bm{z}) \ d \sigma(\bm{z}).$$
\end{definition}

This definition generalizes the curvature invariant introduced in \cite{curv}, see the comment at the end of page 31 of \cite{BES}. Clearly, $0 \le K_{\bfT} \le \rm dim ({\rm Ran} \Delta_{\bfT})$. In the case when $\bfT$ is pure, an appeal to Theorem 4.11 of \cite{BJ} tells us that the  extreme case  $K_{\bfT} = \rm dim ({\rm Ran} \Delta_{\bfT})$ occurs if and only if  $\bfT$ is unitarily equivalent to $(M_{z_1} \otimes I,  \ldots , M_{z_d} \otimes I)$ on $H_s \otimes {\rm Ran} \Delta_{\bfT}$.

\medskip

If two $1/s$-contractions $\bfT$ and $\bfR$ are componentwise intertwined by a unitary, then  the corresponding characteristic functions {\em coincide} in the sense that $\tilde{\theta_{\bfT}} (\bm{z}) = U \tilde{\theta_{\bfR}} (\bm{z})V$ for appropriate unitaries, see \cite{BJ}. Thus, $K_{\bfT}$ is a unitary invariant. In fact, $K_T =- \rm index (T)$ when $d=1$, the kernel $s$ is the Szego kernel and $T$ is a pure contraction ($\|T\| \leq 1$ and $T^{*n} \rightarrow 0 $ strongly as $n \rightarrow \infty$), see \cite{Parrott}. Parrott actually proves in \cite{Parrott} that for a pure contraction, $K_T =\rm dim ({\rm Ran} \Delta_T) -\rm dim (\mathcal D_T)$. While Parrott has a long proof, the proof is immediate using the characteristic function (our characteristic function identifies with the classical one due to Sz.-Nagy and Foias in this case). Indeed,  $\theta_T$ is an inner function implying that $\theta_T(z) \theta_T(z)^*$ is the projection onto  $\mathcal D_T$ almost everywhere  $\bm{z} \in \mathbb T$. That finishes the proof.

\section{The asymptotic formula for multipliers}

An asymptotic formula for the curvature invariant can be obtained from an asymptotic formula for $\int_{\partial \mathbb{B}_{d}} {\rm trace}( \tilde{\theta_{\bfT}}(\bm{z}) \tilde{\theta_{\bfT}} (\bm{z})^{*}  )\ d \sigma(\bm{z})$. The latter will be achieved through a trace class operator defined below in \eqref{d Gamma}.

\medskip
	
Let $H^{2}(\partial \mathbb{B}_{d})$ be the Hardy space of the unit ball $\mathbb{B}_{d}.$ Monomials form an orthogonal basis for $H^{2}(\partial \mathbb{B}_{d}).$ Moreover, for any $\alpha\in\mathbb{Z}^{d}_{+},$
$$\|z^{\alpha}\|^{2}_{H^{2}(\partial \mathbb{B}_{d})} = \frac{1}{\binom{|\alpha| +d-1}{d-1}} \frac{1}{\binom{|\alpha|}{\alpha}} = \frac{1}{q_{d-1}(|\alpha|)} \frac{1}{\binom{|\alpha|}{\alpha}}$$
where $q_{m}(n) = \binom{m+n}{m}$ for any two non-negative integers $m$ and $n$ with the convention that $\binom{0}{0} = 1.$ Since $a_{n} \leq 1 \leq q_{d-1}(n)$ for all $n \geq 0,$ it follows that any unitarily invariant CNP space $H_{s}$ is contractively contained in $H^{2}(\partial \mathbb{B}_{d}).$ Let
$$\delta: H_{s} \otimes {\rm Ran}\Delta_{\bfT} \to H^{2}(\partial \mathbb{B}_{d}) \otimes {\rm Ran}\Delta_{\bfT} $$
be the inclusion map. Let $E_{n}$ and $\tilde{E_{n}}$ be the projections of $H_{s} \otimes {\rm Ran}\Delta_{\bfT}$ and $H^{2}(\partial \mathbb{B}_{d}) \otimes {\rm Ran}\Delta_{\bfT}$ respectiely, onto the subspaces of vector-valued homogenous polynomials of degree $n.$ Let $P_{n} = \sum\limits_{i=0}^{n} E_{i}$ and $\tilde{P_{n}} = \sum\limits_{i=0}^{n} \tilde{E_{i}}.$ The inclusion map has the following properties.
	
	\begin{lemma}\label{delta}
	The inclusion map $\delta: H_{s} \otimes {\rm Ran} \Delta_{\bfT} \to H^{2}(\partial \mathbb{B}_{d}) \otimes {\rm Ran} \Delta_{\bfT}$ has the following properties:
	\begin{enumerate}
	\item $\delta E_{n} = \tilde{E_{n}} \delta, \quad n=0,1,2,\dots,$
	\item $\delta (M_{z_{i}} \otimes I) = (M_{z_{i}} \otimes I) \delta, \quad i=1,\dots,d,$ and
	\item $ \delta^{*} \delta = \sum\limits_{n=0}^{\infty} \frac{a_{n}}{q_{d-1}(n)} E_{n}.$
	\end{enumerate}
	\end{lemma}
	\begin{proof}
	$(1)$ and $(2)$ follow directly from the definition of $\delta.$ To prove $(3),$ fix $\alpha \in \mathbb{Z}^{d}_{+}$ and $h \in {\rm Ran}\Delta_{\bfT}$ and consider
	\begin{align*}
	\left\la \delta^{*} \delta ( \bm{z}^{\alpha} \otimes h), (\bm{z}^{\alpha} \otimes h) \right\ra_{H_{s} \otimes {\rm Ran} \Delta_{\bfT}}
	& = \left\la \delta (\bm{z}^{\alpha} \otimes h), \delta ( \bm{z}^{\alpha} \otimes h) \right\ra_{H^{2}(\partial \mathbb{B}_{d}) \otimes {\rm Ran} \Delta_{\bfT}} \\
	& = \left\| (\bm{z}^{\alpha} \otimes h ) \right\|^{2}_{H^{2}(\partial \mathbb{B}_{d}) \otimes {\rm Ran} \Delta_{\bfT}}\\
	& = \frac{\| h \|^{2}}{q_{d-1}(\alpha) \binom{|\alpha|}{\alpha}} \\
	& = \frac{a_{|\alpha|}}{q_{d-1}(|\alpha|)} \left\| ( \bm{z}^{\alpha} \otimes h)\right\|^{2}_{H_{s} \otimes {\rm Ran} \Delta_{\bfT}}\\
	& =  \frac{a_{|\alpha|}}{q_{d-1}(|\alpha|)} \left\la  (\bm{z}^{\alpha} \otimes h), ( \bm{z}^{\alpha} \otimes h)  \right\ra.
	\end{align*}
	This proves the lemma.
	\end{proof}
	
	Let  $ \Phi: \cB(H_{s} \otimes{\rm Ran} \Delta_{\bfT}) \to \cB(H_{s} \otimes {\rm Ran} \Delta_{\bfT})$ be the completely positive map defined as
	$$\Phi (X) = \sum_{\alpha\in\mathbb{Z}^{d}_{+} \backslash \{0\}} b_{\alpha} (\bfM_{\bm{z}}^{\alpha}  \otimes I) X (\bfM_{\bm{z}}^{\alpha} \otimes I)^{*}.$$
	Convergence of the series above follows from the fact that the canonical $d-$tuple $$\bfM_{\bm{z}} \bydef (M_{z_{1}}, \dots, M_{z_{d}})$$ on $H_{s}$ is a $1/s-$contraction. Let $\Psi: \cB (H_{s} \otimes {\rm Ran} \Delta_{\bfT}) \to \cB(H^{2}(\partial \mathbb{B}_{d}) \otimes {\rm Ran} \Delta_{\bfT})$ be another completely positive map given by
	$$  \Psi(X) = \delta X \delta^{*}.$$
	We define a linear map $d\Psi : \cB (H_{s} \otimes {\rm Ran} \Delta_{\bfT}) \to \cB(H^{2}(\partial \mathbb{B}_{d}) \otimes {\rm Ran} \Delta_{\bfT})$ as follows
	$$d \Psi (X) = \Psi(X) - \sum\limits_{\alpha\in\mathbb{Z}^{d}_{+} \backslash \{0\}} b_{\alpha} (\bfM_{\bm z}^{\alpha} \otimes I) \Psi(X) (\bfM_{\bm z}^{\alpha} \otimes I)^{*}.$$
	In the setting of Drury-Arveson kernel, Arveson viewed the map $d \Psi$ as ``differential" of the map $\Psi.$ Note that the map $d \Psi$ can equivalently be defined as
	\begin{equation}\label{d Gamma}
	d \Psi (X) = \delta (X - \Phi(X)) \delta^{*}.
	\end{equation}
    We find it convenient to work with this definition of $d \Psi.$ For all $n=0,1,2,\dots$ and $i=0,\dots,n,$ we set the following notation:
	\begin{equation}\label{weights}
	 w_{i,n} = \begin{cases} a_{i} \left( 1 - \sum\limits_{j=1}^{n-i} b_{j}\right), & 0 \leq i \leq n-1 \\ a_{n}, & i =n.  \end{cases} 
	 \end{equation}
If $\varphi \in {\rm Mult}(H_{s} \otimes \cE, H_{s} \otimes {\rm Ran} \Delta_{\bfT})$ is a  multiplier, then it is an operator-valued bounded analytic function. Thus, the radial limits
$$	\tilde{\varphi} (\bm{z}) = \lim\limits_{r \uparrow 1} \varphi(r \bm{z}), \quad  {\rm and} \quad \tilde{\varphi}(\bm{z})^{*} = \lim\limits_{r \uparrow 1} \varphi (r \bm{z})^{*} $$
exist as strong limits of operators for $\sigma-$almost every $\bm{z} \in \partial \mathbb{B}_{d}.$  We are ready to state the first asymptotic formula of the paper.

    \begin{thm}
  \label{formulae}
    Let $\varphi \in {\rm Mult}(H_{s} \otimes \cE, H_{s} \otimes {\rm Ran} \Delta_{\bfT})$ be a  multiplier. Then, the linear operator $d \Psi (M_{\varphi} M_{\varphi}^{*})$ is in trace class. Moreover,
$${\rm trace} (d \Psi (M_{\varphi} M_{\varphi}^{*})) = \int_{\partial \mathbb{B}_{d}} {\rm trace} \left( \tilde{\varphi}(\bm{z}) \tilde{\varphi}(\bm{z})^{*}\right) d \sigma(\bm{z})= \lim\limits_{n \rightarrow \infty} \sum\limits_{i=0}^{n} w_{i,n} \frac{{\rm trace} (M_{\varphi} M_{\varphi}^{*} E_{i})}{q_{d-1}(i)}.$$
 \end{thm}

 Through several intermediate results, the proof of Theorem \ref{formulae} will be achieved at the end of this section. Firstly, we establish that the array of numbers $\{w_{i,n}; 0 \leq i \leq n\}_{n\geq 0}$ defined in \eqref{weights} possesses desirable properties.
	\begin{lemma}\label{w}
		For any $n\geq 0$ we have $\sum\limits_{i=0}^{n} w_{i,n} =1.$ Moreover, for any $k \geq 0,$ the sequence $\left\{ \sum\limits_{i=0}^{k} w_{i,n} \right\}_{n\geq k}$ converges to $0$ as $n\rightarrow \infty.$ Consequently, $\sum\limits_{i=0}^{n} w_{i,n} r_{i} \rightarrow L$ as $n \rightarrow \infty$ for any sequence $\{r_{n}\}_{n=0}^{\infty}$ of non-negative numbers converging to $L.$
	\end{lemma}
	\begin{proof}
		First we shall prove that $\sum\limits_{i=0}^{n} w_{i,n} =1 $ for any $n \geq 0.$ This is clearly true for $n=0.$ For $n \geq 1$ consider
		\begin{align*}
			\sum\limits_{i=0}^{n} w_{i,n} & = a_{n} + \sum\limits_{i=0}^{n-1} a_{i} \left( 1 - \sum\limits_{j=1}^{n-i} b_{j}\right) = \sum\limits_{i=0}^{n} a_{i} - \sum\limits_{i=0}^{n-1} \sum\limits_{j=1}^{n-i} a_{i}b_{j} \\
			& = \sum\limits_{i=0}^{n} a_{i} - \sum\limits_{i=1}^{n} \sum\limits_{j=1}^{i} a_{i-j} b_{j} = \sum\limits_{i=0}^{n} a_{i} - \sum\limits_{i=1}^{n} a_{i} = a_{0} = 1.
		\end{align*}
		This completes the proof of the first part. To prove the next part, let $k \geq 0$ be fixed. For $n >k$ consider
		\begin{align*}
			\sum\limits_{i=0}^{k} w_{i,n} & = \sum\limits_{i=0}^{k}a_{i} \left( 1 - \sum\limits_{j=1}^{n-i} b_{j}\right) =  \sum\limits_{i=0}^{k} \sum\limits_{j > n-i} a_{i} b_{j} \leq \sum\limits_{i=0}^{k} a_{i} \sum\limits_{j>n-k}b_{j}.
		\end{align*}
		Since $k$ is fixed and the series $\sum\limits_{j\geq 1} b_{j}$ converges, it follows that $\sum\limits_{i=0}^{k} w_{i,n} \rightarrow 0$ as $n \rightarrow \infty.$
	\end{proof}


    The following combinatorial identity can be found in equation $(2.6)$ in \cite{Cheng}.
    \begin{lemma} \label{Id2}
    For any $\beta \in \mathbb{Z}^{d}_{+}$ with $|\beta| \leq n$ the following identity holds
    $$ \sum\limits_{|\alpha| = n-|\beta|} \frac{\binom{|\alpha|}{\alpha} \binom{|\beta|}{\beta}}{\binom{|\alpha + \beta| }{\alpha +\beta}} = \frac{q_{d-1}(n)}{q_{d-1}(|\beta|)}.$$
    \end{lemma}

	\begin{proposition}\label{P_n}
	For every $X \in \cB(H_{s} \otimes {\rm Ran} \Delta_{\bfT})$ and $n \geq 0$ we have
	$${\rm trace} (d\Psi (X) \tilde{P_{n}}) = \sum\limits_{i=0}^{n} w_{i,n} \frac{{\rm trace} X E_{i}}{q_{d-1}(i)}.$$
	\end{proposition}
	\begin{proof}
		For any $n\geq 0,$ consider
	\begin{align}
		{\rm trace}(d\Psi (X) \tilde{P_{n}}) & = {\rm trace} (\delta (X - \Phi(X)) \delta^{*} \tilde{P_{n}}) \nonumber \\
		&  = \sum\limits_{i=0}^{n} {\rm trace}((X - \Phi(X)) \delta^{*} \tilde{E_{i}} \delta) \nonumber \\
		& = \sum\limits_{i=0}^{n} {\rm trace}((X - \Phi(X)) \delta^{*} \delta E_{i}) \nonumber \\
		& = \sum\limits_{i=0}^{n} \frac{a_{i}}{q_{d-1}(i)} {\rm trace} ((X - \Phi(X)) E_{i}) \label{aux1}.
	\end{align}
	Let $\{\xi_{j}\}$ be an orthonormal basis for ${\rm Ran}\Delta_{\bfT}.$ For any $i\geq 1,$
	\begin{align*}
	{\rm trace}(\Phi (X) E_{i} )
	& = \sum\limits_{\alpha\in\mathbb{Z}^{d}_{+} \backslash \{0\}}  {\rm trace}  (b_{\alpha} (\bfM_{\bm{z}}^{\alpha} \otimes I) X (\bfM_{\bm{z}}^{\alpha} \otimes I)^{*} E_{i} )  \\
	& = \sum\limits_{\alpha\in\mathbb{Z}^{d}_{+} \backslash \{0\}}  {\rm trace}  (b_{\alpha} (\bfM_{\bm{z}}^{\alpha} \otimes I)^{*} E_{i}   (\bfM_{\bm{z}}^{\alpha} \otimes I) X  ) \\
	& =   \sum\limits_{ |\alpha| =1}^{i} b_{\alpha} \sum\limits_{|\beta| = i - |\alpha|} \frac{a_{\beta}}{a_{\alpha + \beta}}\sum\limits_{j}   \left\la X (e_{\beta} \otimes \xi_{j}), (e_{\beta} \otimes \xi_{j}) \right\ra   \\
	& =   \sum\limits_{|\beta| =0}^{i-1} \sum\limits_{j}   \left\la X (e_{\beta} \otimes \xi_{j}), (e_{\beta} \otimes \xi_{j}) \right\ra  \sum\limits_{|\alpha| = i- |\beta|}  b_{\alpha} \frac{a_{\beta}}{a_{\alpha + \beta}}
	\end{align*}
    Now using Lemma \ref{Id2}, we get that
    $$\sum\limits_{|\alpha| = i- |\beta|}  b_{\alpha} \frac{a_{\beta}}{a_{\alpha + \beta}}  = b_{i-|\beta|} \frac{a_{|\beta|}}{a_{i}} \sum\limits_{|\alpha| = i- |\beta|} \frac{\binom{|\alpha|}{\alpha} \binom{|\beta|}{\beta}}{\binom{|\alpha + \beta|}{\alpha + \beta}} =  b_{i-|\beta|} \frac{a_{|\beta|}}{a_{i}} \frac{q_{d-1}(i)}{q_{d-1}(|\beta|)}.$$
    This implies that for any $i \geq 1,$
    $${\rm trace}(\Phi (X) E_{i} ) = \frac{q_{d-1}(i)}{a_{i}} \sum\limits_{j=0}^{i-1} b_{i-j} \frac{a_{j}}{q_{d-1}(j)} {\rm trace}(X E_{j} ) .$$
    Since $\Phi(X) E_{0} = 0,$ we get that
    \begin{align*}
    \sum\limits_{i=0}^{n} {\rm trace} (\Phi(X)E_{i})
    & = \sum\limits_{i=1}^{n} \sum\limits_{j=0}^{i-1} b_{i-j} \frac{a_{j}}{q_{d-1}(j)} {\rm trace}(X E_{j} ) \\
    & =\sum\limits_{i=0}^{n-1} \sum\limits_{j = 1}^{n-i} b_{j} \frac{a_{i}}{q_{d-1}(i)} {\rm trace}(X E_{i} )
    \end{align*}
 Thus,
	\begin{align*}
			{\rm trace} (d \Psi(X) \tilde{P_{n}})
			& =  \frac{a_{n}}{q_{d-1}(n)} {\rm trace}(X E_{n} ) + \sum\limits_{i=0}^{n-1} a_{i} \left( 1 - \sum\limits_{j=1}^{n-i} b_{j}\right) \frac{{\rm trace} (X E_{i} )}{q_{d-1}(i)}\\
			& = \sum\limits_{i=0}^{n} w_{i,n} \frac{{\rm trace} (X E_{i} )}{q_{d-1}(i)}.
		\end{align*}
		This completes the proof.
	\end{proof}

	If $X \in \cB(H_{s} \otimes {\rm Ran} \Delta_{\bfT})$ is a positive operator, then $\Phi(X)$ and $d\Psi (X)$ both are self-adjoint operators. For such $X,$ by Lemma \ref{P_n} we get that
	\begin{align}\label{Pos X}
		{\rm trace}(d\Psi (X) \tilde{P_{n}}) &= \sum\limits_{i=0}^{n} w_{i,n} \frac{{\rm trace} (X E_{i} )}{q_{d-1} (i)} \leq \sum\limits_{i=0}^{n} w_{i,n} \|X\| \frac{{\rm trace} E_{i} }{q_{d-1}(i)}  \nonumber\\
		& =  {\rm dim}({\rm Ran} \Delta_{\bfT}) \|X\|  \quad \text{(by Lemma \ref{w})}.
	\end{align}

\begin{corollary}\label{trace class}
If $X \in \cB(H_{s} \otimes  {\rm Ran} \Delta_{\bfT})$ is such that $X - \Phi(X)$ is positive, then $d \Psi(X)$ is a positive trace class operator with ${\rm trace}(d \Psi(X)) \leq {\rm dim}({\rm Ran} \Delta_{\bfT}) \|X\|$.
\end{corollary}
\begin{proof}
The positivity of $d \Psi(X)$ follows direcly from the definition. Since the projections $\tilde{P_{n}}$ increase to $I$ with increasing $n,$ we conclude from \eqref{Pos X} that
$$	{\rm trace}(d \Psi (X)) = \sup\limits_{n \geq 0} {\rm trace}(d \Psi (X) \tilde{P_{n}}) \leq  {\rm dim}({\rm Ran} \Delta_{\bfT}) \|X\|  <  \infty.$$
This completes the proof.
\end{proof}

\subsection*{Proof of Theorem \ref{formulae}}
A straightforward calculation gives us that $M_{\varphi} M_{\varphi}^{*} - \Phi (M_{\varphi} M_{\varphi}^{*})$ is positive. Thus, by Corollary \ref{trace class}, $d \Psi(M_{\varphi} M_{\varphi}^{*})$ is a trace class operator. Now, we prove the integral formula for ${\rm trace} (d \Psi(M_{\varphi} M_{\varphi}^{*})).$ Define a linear map 
$$A: \cE \to H^{2}( \partial \mathbb{B}_{d}) \otimes {\rm Ran}\Delta_{\bfT}; \quad  \eta \mapsto \delta  \left( M_{\varphi}(1 \otimes \eta)  \right), \quad (\eta \in \cE).$$
The following calculation proves that $d \Psi \left( M_{\varphi} M_{\varphi}^{*}\right) = A A^{*}.$
\begin{align*}
d \Psi \left( M_{\varphi} M_{\varphi}^{*}\right) & = \delta \left(M_{\varphi} M_{\varphi}^{*} - \Phi \left( M_{\varphi} M_{\varphi}^{*}\right) \right) \delta^{*}\\
& = \delta M_{\varphi} M_{\varphi}^{*} \delta^{*} -  \sum\limits_{\alpha\in\mathbb{Z}^{d}_{+} \backslash \{0\}} b_{\alpha} \delta (\bfM_{\bm{z}}^{\alpha} \otimes I) M_{\varphi} M_{\varphi}^{*} (\bfM_{\bm{z}}^{\alpha} \otimes I)^{*}\delta^{*}\\
& = \delta M_{\varphi} (\delta M_{\varphi})^{*} - \delta M_{\varphi} \left( \sum\limits_{\alpha\in\mathbb{Z}^{d}_{+} \backslash \{0\}} b_{\alpha} (\bfM_{\bm{z}}^{\alpha} \otimes I)   (\bfM_{\bm{z}}^{\alpha} \otimes I)^{*} \right) (\delta M_{\varphi})^{*} \\
& = \delta M_{\varphi} \left( E_{0} \otimes I_{\cE}\right) (\delta M_{\varphi})^{*} = A A^{*}.
\end{align*}
Let $\{ \eta_{i}\}$ be an orthonormal basis for $\cE.$ Each $A \eta_{i}$ is an element in $H^{2}( \partial \mathbb{B}_{d}) \otimes {\rm Ran}\Delta_{\bfT}.$ We consider $A \eta_{i}$ as a function from $\mathbb{B}_{d}$ into ${\rm Ran}\Delta_{\bfT}.$ Note that
$$A \eta_{i}(\bm{z}) = \varphi (\bm{z}) \eta_{i}, \quad \bm{z}\in\mathbb{B}_{d}.$$
Hence the boundary values $\tilde{A \eta_{i}}$ of $A \eta_{i}$ are given by
$$ \tilde{ A \eta_{i}} (\bm{z}) = \tilde{\varphi}(\bm{z}) \eta_{i}, \quad \text{$\sigma$-a.e. $\bm{z} \in \partial \mathbb{B}_{d}$}.$$
Now for $\sigma-$almost every $\bm{z} \in \partial \mathbb{B}_{d},$
\begin{align*}
{\rm trace} \left( \tilde{\varphi} (\bm{z}) \tilde{\varphi}(\bm{z})^{*}\right)
& = {\rm trace} \left( \tilde{\varphi}(\bm{z})^{*} \tilde{\varphi}(\bm{z}) \right) \\
& = \sum_{i} \|  \tilde{\varphi}(\bm{z}) \eta_{i} \|^{2} \\
& = \sum_{i} \| \tilde{ A \eta_{i}} (\bm{z}) \|^{2}.
\end{align*}
Thus,
\begin{align*}
\int_{\partial \mathbb{B}_{d}} {\rm trace} \left( \tilde{\varphi} (\bm{z}) \tilde{\varphi} (\bm{z})^{*} \right) d\sigma(\bm{z})
& = \int_{\partial \mathbb{B}_{d}} \sum_{i} \| \tilde{A\eta_{i}} (\bm{z}) \|^{2} d \sigma(\bm{z})\\
& = \sum_{i} \int_{\partial \mathbb{B}_{d}}  \| \tilde{A\eta_{i}} (\bm{z}) \|^{2} d \sigma(\bm{z}) \\
& = \sum_{i} \| A \eta_{i} \|^{2} = {\rm trace}(A^{*}A) = {\rm trace}(A A^{*}) \\
& = {\rm trace}(d \Psi  (M_{\varphi} M_{\varphi}^{*})).
\end{align*}
This completes the proof of the integral formula. Next, we prove the asymptotic formula for ${\rm trace}(d \Psi (M_{\varphi} M_{\varphi}^{*})).$ Since $d\Psi(M_{\varphi} M_{\varphi}^{*})$ is a positive operator and the projections $\tilde{P_{n}}$ increase to $I_{H^{2}(\partial \mathbb{B}_{d}) \otimes {\rm Ran}\Delta_{\bfT}},$ we have
\begin{align*}
{\rm trace}(d \Psi (M_{\varphi}M_{\varphi}^{*})) & = \lim\limits_{n \rightarrow \infty} {\rm trace}(d \Psi (M_{\varphi}M_{\varphi}^{*}) \tilde{P_{n}})\\
& = \lim\limits_{n \rightarrow \infty} \sum\limits_{i=0}^{n} w_{i,n} \frac{{\rm trace} (M_{\varphi}M_{\varphi}^{*} E_{i} ) }{q_{d-1}(i)} \quad \text{(by Proposition \ref{P_n})}.
\end{align*}

\section{The  asymptotic formula for the curvature invariant}
In an attempt to improve the asymptotic formula obtained in the last section, consider again a $\varphi \in {\rm Mult}(H_{s} \otimes \cE, H_{s} \otimes {\rm Ran} \Delta_{\bfT}).$ We have a conjecture, viz.,
$$ {\rm trace} (d \Psi(M_{\varphi} M_{\varphi}^{*})) = \lim \limits_{n \rightarrow \infty} \frac{{\rm trace} (M_{\varphi} M_{\varphi}^{*} P_{n})}{q_{d}(n)} $$
which is much nicer than the formula in Theorem \ref{formulae}. However, we can prove only one half of it in Lemma \ref{liminf}. When, additionally, either $\varphi$ is inner or $\varphi$ is a polynomial, we have the other half as well. These asymptotic formulae when applied to the characteristic function lead to asymptotic formulae for the curvature invariant in Theorem \ref{thm1}.

\begin{lemma}\label{Fact X}
The series
$$\sum\limits_{\alpha\in\mathbb{Z}^{d}_{+}} a_{\alpha} (\bfM_{\bm z}^{\alpha} \otimes I) (M_{\varphi} M_{\varphi}^{*} - \Phi(M_{\varphi} M_{\varphi}^{*})) (\bfM_{\bm z}^{\alpha}\otimes I)^{*}$$
converges strongly to $M_{\varphi} M_{\varphi}^{*}.$
\end{lemma}
\begin{proof}
We have
\begin{align}
X - \Phi(X) & = M_{\varphi} M_{\varphi}^{*} - \sum\limits_{\alpha\in\mathbb{Z}^{d}_{+} \backslash \{0\}} b_{\alpha} (\bfM_{\bm z}^{\alpha} \otimes I)M_{\varphi}M_{\varphi}^{*} (\bfM_{\bm z}^{\alpha} \otimes I)^{*} \nonumber\\
& = M_{\varphi} \Big( I - \sum\limits_{\alpha\in\mathbb{Z}^{d}_{+} \backslash \{0\}} b_{\alpha} (\bfM_{\bm z}^{\alpha} \otimes I_{\cE})  (\bfM_{\bm z}^{\alpha} \otimes I_{\cE})^{*}\Big) M_{\varphi}^{*} \nonumber \\
& = M_{\varphi} (\Delta_{\bfM_{\bm z}}^{2} \otimes I_{\cE})M_{\varphi}^{*} \label{phiX}.
\end{align} 	
The remaining proof follows from the convergence of the series $\sum\limits_{\alpha\in\mathbb{Z}^{d}{+}} a_{\alpha} \bfM_{\bm z}^{\alpha} \Delta_{\bfM_{\bm z}}^{2} (\bfM_{\bm z}^{\alpha})^{*}$ to the identity operator.
\end{proof}

\begin{lemma}
For any $n\geq 0$,
$$\frac{{\rm trace}(M_{\varphi} M_{\varphi}^{*} E_{n})}{q_{d-1}(n)} = \sum\limits_{i=0}^{n} \frac{a_{n-i}}{a_{n}}\frac{a_{i}}{q_{d-1}(i)} {\rm trace}((M_{\varphi} M_{\varphi}^{*} - \Phi(M_{\varphi} M_{\varphi}^{*})) E_{i}).$$
\end{lemma}
\begin{proof}
Let $M_{\varphi} M_{\varphi}^{*} = X.$ Using Lemma \ref{Fact X} for any $n \geq 0$ we write
\begin{align*}
{\rm trace}(X E_{n})  & = \sum\limits_{|\alpha| \leq n} {\rm trace} \left( a_{\alpha} (\bfM_{\bm z}^{\alpha} \otimes I) (X - \Phi(X)) (\bfM_{\bm z}^{\alpha} \otimes I)^{*} E_{n} \right)\\
& = \sum\limits_{|\alpha| \leq n} {\rm trace} \left(a_{\alpha} (X - \Phi(X)) (\bfM_{\bm z}^{\alpha} \otimes I)^{*} E_{n} (\bfM_{\bm z}^{\alpha} \otimes I) \right).
\end{align*}
Let $\{\xi_j\}$ be an orthonormal basis for ${\rm Ran} \Delta_{\bfT}$. Then
\begin{align*}
& {\rm trace}(X E_{n}) \\
= & \sum\limits_{|\alpha| \leq n} a_{\alpha} \sum\limits_{|\beta| = n-|\alpha|} \sum\limits_{j} \left\la (X - \Phi(X)) (\bfM_{\bm z}^{\alpha} \otimes I)^{*} E_{n} (\bfM_{\bm z}^{\alpha} \otimes I)(e_{\beta} \otimes \xi_{j}), (e_{\beta} \otimes \xi_{j}) \right\ra \\
= & \sum\limits_{|\alpha| \leq n} a_{\alpha} \sum\limits_{|\beta| = n-|\alpha|} \sum\limits_{j} \frac{a_{\beta}}{a_{\alpha + \beta}} \left\la (X- \Phi(X))(e_{\beta} \otimes \xi_{j}), (e_{\beta} \otimes \xi_{j}) \right\ra \\
= & \sum\limits_{|\beta| \leq n} \sum\limits_{j} \left\la (X- \Phi(X))(e_{\beta} \otimes \xi_{j}), (e_{\beta} \otimes \xi_{j}) \right\ra \sum\limits_{|\alpha| = n-|\beta|}   \frac{a_{\alpha} a_{\beta}}{a_{\alpha + \beta}}
\end{align*}
Now we use Lemma \ref{Id2} to get that
\begin{align*}
{\rm trace}(X E_{n})= &  \sum\limits_{i=0}^{n} \frac{a_{n-i} a_{i}}{a_{n}} \frac{q_{d-1}(n)}{q_{d-1}(i)}\sum\limits_{|\beta| = i} \sum\limits_{j} \left\la (X- \Phi(X))(e_{\beta} \otimes \xi_{j}), (e_{\beta} \otimes \xi_{j}) \right\ra \\
= &  \sum\limits_{i=0}^{n} \frac{a_{n-i} a_{i}}{a_{n}} \frac{q_{d-1}(n)}{q_{d-1}(i)} {\rm trace} ((X - \Phi(X))E_{i}).
\end{align*}
This implies that
$$\frac{{\rm trace}(X E_{n})}{q_{d-1}(n)} = \sum\limits_{i=0}^{n} \frac{a_{n-i}}{a_{n}}\frac{a_{i}}{q_{d-1}(i)} {\rm trace}((X - \Phi(X)) E_{i}).$$
\end{proof}
	
\begin{lemma}\label{aux2}
 If $\varphi(\bm z) = \sum\limits_{\alpha\in\mathbb{Z}^{d}_{+}} A_{\alpha} \bm z^{\alpha},$ then for any $n\geq 0$,
$${\rm trace}((M_{\varphi} M_{\varphi}^{*}-\Phi(M_{\varphi} M_{\varphi}^{*})) E_{n}) = \sum\limits_{|\alpha| = n} \frac{1}{a_{\alpha}} {\rm trace}(A_{\alpha} A_{\alpha}^{*}).$$
\end{lemma}
\begin{proof}
Using \eqref{phiX} for any $n \geq 0$ we write
$$ (M_{\varphi} M_{\varphi}^{*} - \Phi(M_{\varphi} M_{\varphi}^{*}))E_{n} = M_{\varphi} (\Delta_{\bfM_{\bm z}}^{2} \otimes I_{\cE}) M_{\varphi}^{*} E_{n}.$$
This implies that
\begin{align*}
{\rm trace}((M_{\varphi} M_{\varphi}^{*}-\Phi(M_{\varphi} M_{\varphi}^{*})) E_{n})
& = {\rm trace} (M_{\varphi} (\Delta_{\bfM_{\bm z}}^{2} \otimes I_{\cE}) M_{\varphi}^{*} E_{n})\\
& = {\rm trace} (M_{\varphi}^{*} E_{n} M_{\varphi} (\Delta_{\bfM_{\bm z}}^{2} \otimes I_{\cE})).
\end{align*}
Let $\{ \eta_{i}\}$ be an orthonormal basis for $\cE.$ Then
\begin{align*}
{\rm trace}((M_{\varphi} M_{\varphi}^{*}-\Phi(M_{\varphi} M_{\varphi}^{*})) E_{n})
& = \sum\limits_{i} \left\la E_{n} M_{\varphi} (1 \otimes \eta_{i}), E_{n} M_{\varphi} (1 \otimes \eta_{i}) \right\ra	\\
& = \sum\limits_{i} \sum\limits_{|\alpha| =n} \left\| \bm z^{\alpha} \otimes A_{\alpha} \eta_{i}  \right\|^{2}\\
& =  \sum\limits_{|\alpha| =n} \frac{1}{a_{\alpha}} {\rm trace}(A_{\alpha} A_{\alpha}^{*}).
\end{align*}
\end{proof}

\begin{corollary}\label{trace XE_n}
If $\varphi(\bm z) = \sum\limits_{\alpha\in\mathbb{Z}^{d}_{+}} A_{\alpha} \bm z^{\alpha},$ then for any $n\geq 0$,
$$\frac{{\rm trace}(M_{\varphi} M_{\varphi}^{*} E_{n})}{q_{d-1}(n)} = \sum\limits_{i=0}^{n} \sum\limits_{|\alpha| =i} \frac{a_{n-i}}{a_{n}}\frac{1}{q_{d-1}(i)}  \frac{{\rm trace}(A_{\alpha} A_{\alpha}^{*})} {\binom{i}{\alpha}}.$$
	\end{corollary}

\begin{lemma}
If $\varphi(\bm z) = \sum\limits_{\alpha\in\mathbb{Z}^{d}_{+}} A_{\alpha} \bm z^{\alpha},$ then
$${\rm trace} (d \Psi(M_{\varphi} M_{\varphi}^{*})) = \sum\limits_{i=0}^{\infty} \sum\limits_{|\alpha| =i} \frac{1}{q_{d-1}(i)}  \frac{{\rm trace} (A_{\alpha} A_{\alpha}^{*})}{\binom{i}{\alpha}} < \infty.$$
\end{lemma}
\begin{proof}
Using \eqref{aux1} for any $n\geq 0$ we write
$${\rm trace} (d \Psi(M_{\varphi} M_{\varphi}^{*}) \tilde{P_{n}}) = \sum\limits_{i=0}^{n} \frac{a_{i}}{q_{d-1}(i)} {\rm trace} (( M_{\varphi} M_{\varphi}^{*} - \Phi(M_{\varphi} M_{\varphi}^{*}))E_{i}).$$
Now, using Lemma \ref{aux2} we get that for any $n \geq 0,$
\begin{align*}
{\rm trace} (d \Psi(M_{\varphi} M_{\varphi}^{*}) \tilde{P_{n}}) 
& =  \sum\limits_{i=0}^{n} \frac{a_{i}}{q_{d-1}(i)} \sum\limits_{|\alpha| =i} \frac{1}{a_{\alpha}} {\rm trace}(A_{\alpha} A_{\alpha}^{*}) \\
& =  \sum\limits_{i=0}^{n} \sum\limits_{|\alpha| =i} \frac{1}{q_{d-1}(i)}  \frac{1}{\binom{i}{\alpha}} {\rm trace}(A_{\alpha} A_{\alpha}^{*})
\end{align*}
By Corollary \ref{trace class}, we know that $d \Psi(M_{\varphi} M_{\varphi}^{*})$ is a positive trace class operator. Since the projections $\tilde{P_{n}}$ converge increasingly to the identity operator, we get that ${\rm trace} (d \Psi(X) \tilde{P_{n}})$ increases to ${\rm trace}(d \Psi(M_{\varphi} M_{\varphi}^{*})).$ This completes the proof.
\end{proof}

\begin{lemma} \label{liminf}
We have
$$\liminf\limits_{n \rightarrow \infty} \frac{{\rm trace} (M_{\varphi} M_{\varphi}^{*} P_{n})}{q_{d}(n)} \geq  \liminf\limits_{n \rightarrow \infty} \frac{{\rm trace} (M_{\varphi} M_{\varphi}^{*}E_{n})}{q_{d-1}(n)} \geq {\rm trace} (d \Psi(M_{\varphi} M_{\varphi}^{*})) .$$
\end{lemma}
\begin{proof}
The first inequality follows from Lemma 3.22 in \cite{curv}. Now we shall prove the second inequality. Let $\varphi(\bm z) = \sum\limits_{\alpha\in\mathbb{Z}^{d}_{+}} A_{\alpha} \bm z^{\alpha}$. Let $\epsilon > 0$ be given. Then there exists a positive integer $n_{0}$ such that
$$ \sum\limits_{i=n_{0}}^{\infty} \sum\limits_{|\alpha| =i} \frac{1}{q_{d-1}(i)}  \frac{{\rm trace} (A_{\alpha} A_{\alpha}^{*})}{\binom{i}{\alpha}}  < \epsilon.$$
For any $n \geq n_{0}$, we get that
$$ \frac{{\rm trace} (XE_{n})}{q_{d-1}(n)}  \geq  \sum\limits_{i=0}^{n_{0}} \sum\limits_{|\alpha| =i} \frac{a_{n-i}}{a_{n}}\frac{1}{q_{d-1}(i)}  \frac{{\rm trace}(A_{\alpha} A_{\alpha}^{*})} {\binom{i}{\alpha}} .$$
Now, using the fact that $\lim\limits_{n \rightarrow \infty} \frac{a_{n}}{a_{n+1}} =1,$ we get the following:
$$\liminf\limits_{n \rightarrow \infty} \frac{{\rm trace} (XE_{n})}{q_{d-1}(n)} \geq {\rm trace}(d \Psi(X)) - \epsilon .$$
Since this holds for any $\epsilon >0,$ we get the desired inequality.
\end{proof}

In the following theorem, we prove that if $\varphi$ is a polynomial, then the inequalities in Lemma \ref{liminf} become equalities and $\liminf$ become $\lim.$
\begin{thm}
Let $\varphi \in {\rm Mult}(H_{s} \otimes \cE, H_{s} \otimes {\rm Ran} \Delta_{\bfT})$ be a polynomial, i.e., there exists a non-negative integer $N$ such that $\varphi(\bm z) = \sum\limits_{|\alpha| \leq N} A_{\alpha} \bm z^{\alpha}.$ Then 
$${\rm trace} (d \Psi (M_{\varphi} M_{\varphi}^{*})) = \lim\limits_{n \rightarrow \infty} \frac{{\rm trace} (M_{\varphi} M_{\varphi}^{*}E_{n})}{q_{d-1}(n)}  = \lim\limits_{n \rightarrow \infty} \frac{{\rm trace} (M_{\varphi} M_{\varphi}^{*} P_{n})}{q_{d}(n)}.$$
\end{thm}
\begin{proof}
By Corollary \ref{trace XE_n}, for $n \geq N,$ we get that
$$\frac{{\rm trace}(M_{\varphi} M_{\varphi}^{*} E_{n})}{q_{d-1}(n)} = \sum\limits_{i=0}^{N} \sum\limits_{|\alpha| =i} \frac{a_{n-i}}{a_{n}}\frac{1}{q_{d-1}(i)}  \frac{{\rm trace}(A_{\alpha} A_{\alpha}^{*})} {\binom{i}{\alpha}}.$$
Now, using the fact that $\lim\limits_{n \rightarrow \infty} \frac{a_{n}}{a_{n+1}} =1,$ we get the following:
$$\lim\limits_{n \rightarrow \infty}\frac{{\rm trace}(M_{\varphi} M_{\varphi}^{*} E_{n})}{q_{d-1}(n)} = \sum\limits_{i=0}^{N} \sum\limits_{|\alpha| =i} \frac{1}{q_{d-1}(i)}  \frac{{\rm trace}(A_{\alpha} A_{\alpha}^{*})} {\binom{i}{\alpha}} = {\rm trace}(d \Psi(M_{\varphi} M_{\varphi}^{*})).$$
The equality 
$$\lim\limits_{n \rightarrow \infty} \frac{{\rm trace} (M_{\varphi} M_{\varphi}^{*}E_{n})}{q_{d-1}(n)}  = \lim\limits_{n \rightarrow \infty} \frac{{\rm trace} (M_{\varphi} M_{\varphi}^{*} P_{n})}{q_{d}(n)}$$
follows from Lemma 3.22 in \cite{curv}.
\end{proof}

An operator valued function $\varphi \in {\rm Mult}(H_{s} \otimes \cE, H_{s} \otimes \mathcal F )$  is called an {\em inner multiplier} if the multiplication operator $M_{\varphi} : H_{s} \otimes \cE \to H_{s} \otimes \mathcal F$ is a partial isometry or equivalently, $M_{\varphi} M_{\varphi}^{*}$ is an orthogonal projection. 
As mentioned at the beginning of this section, we would like to make the inequality in Lemma \ref{liminf} an equality and $\liminf$ a $\lim.$ This can be done for an inner multiplier. In that case, it also relates to another concept which we define below.

\begin{definition}
Let $\cM \subseteq H_{s} \otimes {\rm Ran} \Delta_{\bfT},$ and $\cM(\lambda) = \{ f(\lambda): f \in \cM\} \subset {\rm Ran} \Delta_{\bfT}$ for any $\lambda \in \mathbb{B}_{d}.$ Then the fibre dimension of $\cM$ is defined by
$$fd(\cM) = \sup\limits_{\lambda \in \mathbb{B}_{d}} {\rm dim} \cM(\lambda).$$
\end{definition}

The following remarkable theorem was proved in \cite{GRS}.

\begin{thm}[Greene--Richter--Sundberg]\label{GRS}
If $\varphi \in {\rm Mult}(H_{s} \otimes \cE, H_{s} \otimes {\rm Ran}\Delta_{\bfT})$ is an inner multiplier, then $\tilde{\varphi}(\bm z)$ is a partial isometry with ${\rm rank}(\tilde{\varphi}(\bm z)) = fd({\rm Ran} M_{\varphi})$ for $\sigma$ a. e. $\partial \mathbb{B}_{d}.$
\end{thm}

\begin{thm}\label{innermult}
Let $\varphi \in {\rm Mult} (H_{s} \otimes \cE , H_{s} \otimes {\rm Ran}\Delta_{\bfT})$ be an inner multiplier. Then
$${\rm trace} (d \Psi(M_{\varphi} M_{\varphi}^{*})) = fd({\rm Ran} M_{\varphi}) = \lim\limits_{n \rightarrow \infty} \frac{{\rm trace} (M_{\varphi} M_{\varphi}^{*}P_{n})}{q_{d}(n)}.$$
\end{thm}
\begin{proof}
Using Theorem \ref{formulae} we write
$${\rm trace} (d \Psi (M_{\varphi} M_{\varphi}^{*})) = \int\limits_{\partial \mathbb{B}_{d}} {\rm trace} (\tilde{\varphi} (\bm z) \tilde{\varphi}(\bm z)^{*}) d \sigma(\bm z).$$
Since $\varphi$ is inner, we apply Theorem \ref{GRS} to get that
$$ fd({\rm Ran} M_{\varphi}) = \int\limits_{\partial \mathbb{B}_{d}} {\rm rank}(\tilde{\varphi} (\bm z)) d \sigma(\bm z) =  \int\limits_{\partial \mathbb{B}_{d}} {\rm trace} (\tilde{\varphi} (\bm z) \tilde{\varphi}(\bm z)^{*}) d \sigma(\bm z).$$
Note that ${\rm Ran}M_{\varphi}$ is a closed subspace of $H_{s} \otimes {\rm Ran} \Delta_{\bfT}$ which is invariant under $M_{z_{i}} \otimes I$ for all $i=1,\dots,d.$ By Lemma 4 in \cite{Fang-MRL}, we have
$$ fd ({\rm Ran} M_{\varphi}) = \lim\limits_{n \rightarrow \infty} \frac{{\rm dim} (P_{n} {\rm Ran} M_{\varphi})}{q_{d}(n)}.$$
It is easy to see that for any $n \geq 0,$
$${\rm dim}(P_{n} {\rm Ran} M_{\varphi}) \geq | {\rm trace} (P_{n}P_{\cM})|.$$
By Corollary \ref{trace XE_n},  ${\rm trace} (P_{n} P_{\cM}) = {\rm trace}(P_{\cM} P_{n}) $ is positive. Thus, for any $n \geq 0,$
$${\rm dim}(P_{n} {\rm Ran}M_{\varphi}) \geq {\rm trace} (P_{{\rm Ran} M_{\varphi}} P_{n}) = {\rm trace} (M_{\varphi} M_{\varphi}^{*} P_{n}).$$
This implies that
$$\limsup\limits_{n \rightarrow \infty} \frac{{\rm trace} (M_{\varphi} M_{\varphi}^{*} P_{n})}{q_{d}(n)} \leq fd({\rm Ran M_{\varphi}}) = {\rm trace}(d \Psi(M_{\varphi} M_{\varphi}^{*}))$$
Now using Lemma \ref{liminf} we conclude that
$${\rm trace}(d \Psi(M_{\varphi} M_{\varphi}^{*})) = fd({\rm Ran M_{\varphi}}) = \lim\limits_{n \rightarrow \infty} \frac{{\rm trace} (M_{\varphi} M_{\varphi}^{*} P_{n})}{q_{d}(n)} .$$
\end{proof}

Let $\cM \subseteq H_{s} \otimes \mathbb{C}^{N}$ be closed subspace which is invariant under each $M_{z_{i}}.$ Then by a result of McCullough--Trent \cite{MT}, there exists a Hilbert space $\cE$ and an inner multiplier $\varphi \in {\rm Mult} (H_{s} \otimes \cE, H_{s} \otimes \mathbb{C}^{N})$ such that ${\rm Ran} M_{\varphi} = \cM.$ Now by using Theorem \ref{innermult}, we get that
$$fd(\cM) = \lim\limits_{n \rightarrow \infty} \frac{{\rm trace} (P_{\cM} P_{n})}{q_{d}(n)}.$$
when $s$ is the Dirichlet kernel, this result appeared in \cite{Fang-Crelle}.

\medskip

It is well-known that for a  $1/s$-contraction $\bfT = (T_{1}, \hdots,T_{d})$, the series
$$\sum\limits_{\alpha\in\mathbb{Z}^{d}_{+}} a_{\alpha} \bfT^{\alpha} \Delta_{\bfT}^{2} (\bfT^{\alpha})^{*}$$
converges in the strong operator topology and the limiting operator is a positive contraction, see Lemma 4.1 in \cite{BJ}. If the series
$$\sum\limits_{\alpha\in\mathbb{Z}^{d}_{+}} a_{\alpha} \bfT^{\alpha} \Delta_{\bfT}^{2}(\bfT^{\alpha})^{*}$$ converges strongly to $I$,
the $1/s-$contraction $\bfT$ is called $pure$. When $s$ is the Szeg\H{o} kernel  (respectively when $s$ is the Drury-Arveson kernel), this notion of  pureness is exactly the same as that of a $C_{\cdot 0}$ contraction (respectively that of a pure $d$-contraction). It is known that $\theta_{\bfT}$ is an inner multiplier if $\bfT$ is {\em pure}.

\medskip

Finally, we have the asymptotic formula for the curvature invariant. The proof is omitted because it follows from the results obtained above.
\begin{thm}\label{thm1}
Let $\bfT = (T_{1}, \dots,T_{d})$ be a $1/s-$contraction.  Then
\begin{align*} K_{\bfT} & =  {\rm dim} ({\rm Ran}\Delta_{\bfT}) - {\rm trace}(d \Psi (M_{\theta_{\bfT}} M_{\theta_{\bfT}}^{*})) \\
& = {\rm dim} ({\rm Ran}\Delta_{\bfT}) - \lim\limits_{n\rightarrow \infty} \sum\limits_{i=0}^{n} w_{i,n} \frac{{\rm trace} (M_{\theta_{\bfT}} M_{\theta_{\bfT}}^{*} E_{i} )}{q_{d-1}(i)} .\end{align*}
Moreover, if $\bfT$ is a pure $1/s-$contraction, then $K_{\bfT}$ is an integer and 
$$ K_{\bfT} = {\rm dim} ({\rm Ran}\Delta_{\bfT}) - \lim\limits_{n\rightarrow \infty}  \frac{{\rm trace} (M_{\theta_{\bfT}} M_{\theta_{\bfT}}^{*} P_{n} )}{q_{d}(n)} 
= {\rm dim} ({\rm Ran}\Delta_{\bfT}) -  fd({\rm Ran} M_{\theta_{\bfT}}).$$
Furthermore, if $\bfT$ is a $1/s-$contraction such that its characteristic function $\theta_{\bfT}$ is a polynomial, then 
$$ K_{\bfT}  = {\rm dim} ({\rm Ran}\Delta_{\bfT}) - \lim\limits_{n\rightarrow \infty} \frac{{\rm trace} (M_{\theta_{\bfT}} M_{\theta_{\bfT}}^{*} E_{i} )}{q_{d-1}(i)}  
 = {\rm dim} ({\rm Ran}\Delta_{\bfT}) - \lim\limits_{n\rightarrow \infty}  \frac{{\rm trace} (M_{\theta_{\bfT}} M_{\theta_{\bfT}}^{*} P_{n} )}{q_{d}(n)}.$$
\end{thm}


Acknowledgments.
The first author is supported by a J C Bose Fellowship JCB/2021/000041 of SERB. The second author is supported by the Prime Minister’s Research Fellowship PM/MHRD-20-15227.03. This work is also supported by the DST FIST program - 2021 [TPN - 700661].

\end{document}